\documentclass[a4pape]{tbimc}

\usepackage{mathtools}

\renewcommand{\P}{\mathsf P}
\newcommand*{\abs}[1]{\left\lvert#1\right\rvert}
\newcommand*{\set}[1]{\left\{#1\right\}}
\newcommand{\E}{\mathsf E}
\newcommand{\R}{\mathbb R}
\newcommand{\ind}{\mathbf 1}
\DeclareMathOperator{\cov}{Cov}
\DeclareMathOperator{\var}{Var}

\newtheorem{theorem}{Theorem}[section]

\newtheorem{proposition}[theorem]{Proposition}
\newtheorem{corollary}[theorem]{Corolary}
\theoremstyle{definition}
\newtheorem{definition}[theorem]{Definition}
\theoremstyle{remark}
\newtheorem{remark}[theorem]{Remark}

\allowdisplaybreaks

\begin{document}
\selectlanguage{english}

\title[Fractional Cox-Ingersoll-Ross process]
	{Stochastic Representation and Pathwise Properties of Fractional Cox-Ingersoll-Ross Process}

\author{Yu. Mishura}
\address{Department of Probability Theory, Statistics and Actuarial Mathematics,
	Faculty of Mechanics and Mathematics,
	Taras Shevchenko National University of Kyiv,
	Volodymyrska St., 64/13,
	Kyiv 01601,
	Ukraine}
\email{myus@univ.kiev.ua}

\author{V. Piterbarg}
\address{Laboratory of Probability Theory,
 	Faculty of Mechanics and Mathematics,
  	Moscow State University,
 	Leninskie Gory, 1,
 	Moscow 119991,
 	Russia }
\email{piter@mech.math.msu.su}

\author{K. Ralchenko}
\address{Department of Probability Theory, Statistics and Actuarial Mathematics,
	Faculty of Mechanics and Mathematics,
	Taras Shevchenko National University of Kyiv,
	Volodymyrska St., 64/13,
	Kyiv 01601,
	Ukraine}
\email{k.ralchenko@gmail.com}

\author{A. Yurchenko-Tytarenko}
\address{Department of Probability Theory, Statistics and Actuarial Mathematics,
	Faculty of Mechanics and Mathematics,
	Taras Shevchenko National University of Kyiv,
	Volodymyrska St., 64/13,
	Kyiv 01601,
	Ukraine}
\email{ayurty@gmail.com}	

\date{}

\begin{abstract}
    We consider the fractional Cox--Ingersoll--Ross process  satisfying the stochastic differential equation (SDE)
    $dX_t = aX_t\,dt + \sigma \sqrt{X_t}\,dB^H_t$
    driven by a fractional Brownian motion (fBm) with Hurst parameter exceeding $\frac{2}{3}$.
    The integral $\int_0^t\sqrt{X_s}dB^H_s$ is considered as a pathwise integral and is equal to the limit of Riemann--Stieltjes integral sums.
    It is shown that the fractional Cox--Ingersoll--Ross process is a square of the fractional Ornstein--Uhlenbeck process until the first zero hitting.
    Based on that, we consider the square of the fractional Ornstein--Uhlenbeck process with an arbitrary Hurst index and prove that until its first zero hitting it satisfies the specified SDE if the integral $\int_0^t\sqrt{X_s}\,dB^H_s$ is defined as a pathwise Stratonovich integral.
    Therefore, the question about the first zero hitting time of the Cox\,--\,Ingersoll\,--\,Ross process, which matches the first zero hitting moment of the fractional Ornstein\,--\,Uhlenbeck process, is natural.
    Since the latter is a Gaussian process, it is proved by the estimates for distributions of Gaussian processes that for $a<0$ the probability of hitting zero in finite time is equal to 1, and in case of $a>0$ it is positive but less than 1.
    The upper bound for this probability is given.
\end{abstract}

\maketitle

\section{Introduction}

The standard Cox\,--\,Ingersoll\,--\,Ross diffusion process was introduced and studied in papers \cite{CIR1}--\cite{CIR3} for the purpose of better interest rate modeling and as a generalization of the Vasicek model.

The Cox\,--\,Ingersoll\,--\,Ross process (CIR) is a one-factor model that contains one source of randomness. In this model it is assumed that the instantaneous value of the interest rate $r_t$ is a solution of the stochastic differential equation (SDE)
\[
    dr_t = a(b-r_t)\,dt + \sigma \sqrt{r_t}\,dW_t, \quad t\ge0,
\]
where
$a, b, \sigma \in \R^+$, $W=\{W_t,t\ge0\}$~is a Wiener process, $r|_{t=0}=r_0>0$.
The parameter $a$ corresponds to the speed of adjustment (i.e. convergence to the mean), $b$ is the mean, and $\sigma$~ is the volatility.
If the condition $2ab\ge\sigma^2$ holds, the process takes only positive values with a probability of 1 and does not hit zero.
Unlike the Vasicek model, wherein the standard deviation is a constant, the standard deviation $\sigma\sqrt{r_t}$ of the CIR process is proportional to the value of the process.

The CIR process is ergodic and has a stationary distribution.
The distribution of its future values $r_{t+T}$, provided that $r_t$ is known, is a noncentral chi-square distribution, and the distribution of the limit value $r_{\infty}$ is a gamma distribution.
The CIR process is also widely used in stochastic volatility modeling in the Heston model.
An extensive bibliography on this subject can be found, for example, in \cite{KMM,KIM}.

Together with the above, it should be noted that in reality many financial models demonstrate the so-called ''memory phenomenon'', meaning that price changes in the market cannot be characterized only by randomness generated by the Wiener process.
Regarding the description of financial markets with memory, see, for example, \cite{AI,BM,DGE,YMH}.

Although it was considered that the best models for interest rates and stochastic volatility were those containing a fractional Brownian motion (fBm) with the Hurst index $H>\frac{1}{2}$, recent studies (see, for example, \cite{BFG}) have noted that volatility can be so irregular that the Hurst index is estimated to be $0.1$.

Consequently, to simulate the corresponding interest rates or stochastic volatility, it is useful to introduce the fractional Ornstein-Uhlenbeck process or, in order to preserve the nonnegativity of the random process, the fractional Cox-Ingersoll-Ross process.
While the fractional Ornstein-Uhlenbeck process is Gaussian, so there are no problems with stochastic integration (the properties of the fractional Ornstein-Uhlenbeck process are described in \cite{CKM}), there are several different approaches to the definition of the fractional CIR process. The pathwise integration with respect to fBm is considered for $H>2/3$ in \cite{MMS}, and the so-called ''rough path approach'' is introduced in \cite{Mar}. Another method is based on the fact that the standard CIR process belongs to class of Pearson diffusions, so the fractional CIR process can be defined as a time changed CIR process with inverse stable subordinator (see \cite{LMS1, LMS2}).

In this paper we first consider the SDE
\begin{equation}\label{eq: 1}
	dX_t = aX_t\,dt + \sigma\sqrt{X_t}\,dB_t^H, \quad t\ge0,
\end{equation}
where $X|_{t=0}=x_0>0$, $a\in\R$, $\sigma>0$,
$B^H=\{B_t^H,t\ge0\}$ is a fBm with the Hurst parameter
\mbox{$H \in (0,1)$}, i.e. a centered Gaussian process with the following covariance function: $\E B_t^H B_s^H = \frac{1}{2} \left(t^{2H} + s^{2H} - \abs{t-s}^{2H}\right)$.
In this case, we assume that $H\in(2/3,1)$, because then the integral $\int_0^t\sqrt{X_s}\,dB_s^H$ is defined as a pathwise integral, which is the limit of Riemann\,--\,Stieltjes integral sums.
It is proved that the unique solution of the equation \eqref{eq: 1} is the square of the fractional Ornstein\,--\,Uhlenbeck process until the first moment of hitting zero.
It is clear that, due to the uniqueness of the solution, the solution of the equation \eqref{eq: 1} remains at zero after the first moment of reaching it.

After this, we consider the square of the fractional Ornstein\,--\,Uhlenbeck process with an arbitrary Hurst index $H \in (0,1)$ and prove that this square satisfies the equation \eqref{eq: 1} until the first moment of hitting zero, if the stochastic integral is considered as the pathwise Stratonovich integral.

Naturally, we face the question of the corresponding zero hitting time finiteness.
It is proved that for the positive speed of adjustment the probability of finiteness of this moment is 1, and for the negative speed of adjustment it is between 0 and 1, and an upper bound for this probability is given.
As an auxiliary result, the form of the covariance function of the fractional Ornstein\,--\,Uhlenbeck process is used.

The paper is organized as follows.
In Section \ref{sec1}, equation \eqref{eq: 1} is considered for $H\in(2/3,1)$ with the Riemann-Stieltjes pathwise integral with respect to fBm and it is shown that its solution is the square of the fractional Ornstein\,--\,Uhlenbeck process until hitting zero.
Section \ref{sec2} introduces the fractional Cox\,--\,Ingersoll\,--\,Ross process for an arbitrary $H\in(0,1)$ as the square of the corresponding fractional Ornstein-Uhlenbeck process and shows that the process, defined in this way, satisfies the equation \eqref{eq: 1} with the pathwise Stratonovich integral.
Section \ref{sec3} is devoted to studying the probability of hitting zero in a finite time by the described process.
Section \ref{sec4} contains an auxiliary result --- a formula for the covariance function of the fractional Ornstein-Uhlenbeck process.

\section{Fractional Cox\,--\,Ingersoll\,--\,Ross process for $H\in(2/3, 1)$}\label{sec1}

Consider the stochastic differential equation of the following form:
\begin{equation}\label{main sde}
    dX_t = \tilde a X_t\,dt + \tilde\sigma\sqrt{X_t}\,dB_t^H, \quad t\ge0,
\end{equation}
where $X|_{t=0}=x_0>0$, $\tilde a\in\R$, $\tilde\sigma>0$.

According to Theorem 6 from \cite{MMS}, if $H>2/3$, the equation \eqref{main sde} has a unique solution until the first moment of reaching zero, and the integral
$\int_0^t \sqrt{X_s}\,dB_s^H$
exists as a pathwise Riemann\,--\,Stieltjes sums limit.
This can be explained empirically as follows: the integral with respect to fractional Brownian motion (for the conditions of existence and properties of such integrals see, for example, \cite{MZ}) exists as a pathwise limit of Riemann-Stieltjes sums if the sum of H\"older exponents of integrator and integrand exceeds 1.
On the other hand, provided the solution exists, it is H\"older-continuous up to order $H$ (see, for example, \cite{feyel}), and hence the integrand function $\sqrt{X_t}$ will be H\"older-continuous up to order $H/2$.
Thus, the condition for the existence of a pathwise integral with respect to fractional Brownian motion and the fact that the integral is the limit of Riemann-Stieltjes sums is the inequality $H/2+H>1$ or $H>2/3$.
We recall that in this case the equation \eqref{main sde} has a unique solution with strictly positive trajectories until the first moment of hitting zero.

Denote
$\tau_0 \coloneqq \inf\{t>0: X_t=0\}$
and consider the trajectories of the process $\{X_t, t\ge0\}$ on $[0,\tau_0)$.
After substitution $Y_t=\sqrt{X_t}$ and using the Ito formula for integrals with respect to fractional Brownian motion (see \cite{mishura}), we obtain
\[
dY_t = \frac{dX_t}{2\sqrt{X_t}} = \frac{\tilde a X_t\,dt}{2\sqrt{X_t}} + \frac{\tilde\sigma}{2}\,dB_t^H.
\]

Denoting $a=\tilde a/2$, $\sigma=\tilde\sigma/2$, we get
\begin{equation}\label{fOU}
dY_t=a\,Y_t\,dt+\sigma\,dB_t^H
\end{equation}
with the initial condition $Y_0=\sqrt{X_0}$.
Thus, the solution $\{X_t, t\in[0,\tau_0)\}$ of the equation \eqref{main sde} is the square of the fractional Ornstein\,--\,Uhlenbeck process (the latter was introduced in \cite{CKM}) until it reaches zero.

\section{Generalization of the fractional Cox\,--\,Ingersoll\,--\,Ross process for $H\in(0, 1)$}
\label{sec2}

Taking into account the conclusions of Section \ref{sec1}, now let us define the fractional Cox\,--\,Ingersoll\,--\,Ross process for all Hurst indices $H\in(0, 1)$.

\begin{definition}\label{def:CIR}
    Let $H\in(0,1)$ be an arbitrary Hurst index, $\{Y_t, t\ge0\}$ be a fractional Ornstein\,--\,Uhlenbeck process, that satisfies equation (\ref{fOU}), and $\tau$ be the first moment of reaching zero by the latter.
    The \emph{fractional Cox\,--\,Ingersoll\,--\,Ross process} is the process $\{X_t, t\ge0\}$ such that for all $t\ge0$
    \begin{equation}
        X_t(\omega) = Y_t^2(\omega) \ind_{\{t<\tau(\omega)\}}.
    \end{equation}
\end{definition}

It should be noted that the process defined in this way satisfies the equation of the form \eqref{main sde}, if we consider the integral $\int_0^t \sqrt{X_s}\,dB_s^H$ as the pathwise Stratonovich integral. Let us give the corresponding definition.

\begin{definition}
    Let $\{X_t, t\ge0\}$, $\{Y_t, t\ge0\}$ be random processes.
    The pathwise \emph{Stratonovich integral} $\int_0^TX_s\circ dY_s$ is a pathwise limit of the following sums
    \begin{equation*}
        \sum_{k=1}^{n} \frac{X_{t_{k}} + X_{t_{k-1}}}{2} \left(Y_{t_{k}} - Y_{t_{k-1}}\right),
    \end{equation*}
    as the mesh of the partition
    $0=t_0<t_1<t_2<\ldots<t_{n-1}<t_n=T$
    tends to zero, in case if this limit exists.
\end{definition}

Indeed, suppose that $\{Y_t,t\ge0\}$ is a fractional Ornstein\,--\,Uhlenbeck process, which starts at $\sqrt{X_0}$, $\tau=\inf\{s>0: Y_s=0\}$, and for some $\omega\in\Omega$ consider the point $t$ such that $t<\tau(\omega)$.
Then, according to the definition ~\ref{def:CIR},
\begin{equation}\label{eq: 2}
    X_t = Y_t^2 = \left(\sqrt{X_0} + a\int_0^tY_s\,ds + \sigma B_t^H\right)^2.
\end{equation}

Consider an arbitrary partititon of the interval $[0,t]$:
\[
    0=t_0<t_1<t_2<\ldots<t_{n-1}<t_n=t.
\]

Using the formula \eqref{eq: 2}, we have
\begin{align*}
    X_t &= \sum_{k=1}^{n} \left(X_{t_k} - X_{t_{k-1}}\right) + X_0
    \\
    &= \sum_{k=1}^n \left(\left[\sqrt{X_0} + a\int_0^{t_k}Y_s\,ds + \sigma B_{t_k}^H\right]^2 - \left[\sqrt{X_0} + a\int_0^{t_{k-1}}Y_s\,ds + \sigma B_{t_{k-1}}^H\right]^2\right) + X_0
    \\
    &= \sum_{k=1}^n \Biggl(2\sqrt{X_0} + a \left(\int_0^{t_k} Y_s\,ds + \int_0^{t_{k-1}} Y_s\,ds\right)
    + \sigma \left(B_{t_k}^H + B_{t_{k-1}}^H\right)\Biggr)\times
    \\
    &\quad\times\left(a \int_{t_{k-1}}^{t_k}Y_s\,ds + \sigma \left(B_{t_k}^H - B_{t_{k-1}}^H\right)\right) + X_0.
\end{align*}

Expanding the brackets in the last expression, we obtain:
\begin{equation}\label{integral sums}
\begin{split}
    X_t &=2 a\sqrt{X_0} \sum_{k=1}^n \int_{t_{k-1}}^{t_k}Y_s\,ds + a^2 \sum_{k=1}^n \left(\int_0^{t_k}Y_s\,ds + \int_0^{t_{k-1}}Y_s\,ds\right) \int_{t_{k-1}}^{t_k}Y_s\,ds
    \\
    &\quad+a\sigma \sum_{k=1}^n \left(B_{t_k}^H + B_{t_{k-1}}^H\right) \int_{t_{k-1}}^{t_k}Y_s\,ds + 2\sigma\sqrt{X_0} \sum_{k=1}^n \left(B_{t_k}^H - B_{t_{k-1}}^H\right)
    \\
    &\quad+ a\sigma \sum_{k=1}^n \left(\int_0^{t_k}Y_s\,ds + \int_0^{t_{k-1}}Y_s\,ds\right) \left(B_{t_k}^H - B_{t_{k-1}}^H\right)
    \\
    &\quad+ \sigma^2 \sum_{k=1}^n \left(B_{t_k}^H
    + B_{t_{k-1}}^H\right) \left(B_{t_k}^H - B_{t_{k-1}}^H\right).
\end{split}
\end{equation}

Let the mesh $\Delta t$ of the partition tend to zero. The first three terms
\begin{equation}\label{lebesgue sums}
\begin{split}
&2a\sqrt{X_0} \sum_{k=1}^n \int_{t_{k-1}}^{t_k}Y_s\,ds + a^2\sum_{k=1}^n \left(\int_0^{t_k}Y_sds + \int_0^{t_{k-1}}Y_s\,ds\right) \int_{t_{k-1}}^{t_k}Y_s\,ds
\\
&\quad\quad+ a\sigma \sum_{k=1}^n \left(B_{t_k}^H + B_{t_{k-1}}^H\right) \int_{t_{k-1}}^{t_k}Y_s\,ds
\\
&\quad\rightarrow 2a\sqrt{X_0}\int_0^tY_s\,ds + 2a^2\int_0^tY_s\int_0^sY_u\,du\,ds + 2a\sigma \int_0^tB_s^HY_s\,ds
\\
&\quad= 2a\int_0^t Y_s \left(\sqrt{X_0} + a\int_0^sY_u\,du + \sigma B_s^H\right)ds
= 2a\int_0^t Y^2_s\,ds
\\
&\quad=2a\int_0^tX_s\,ds
=\tilde a\int_0^tX_s\,ds, \quad \Delta t\rightarrow 0,
\end{split}
\end{equation}
and the last three terms
\begin{equation}\label{stratonovich sums}
\begin{split}
&2\sigma\sqrt{X_0}\sum_{k=1}^n \left(B_{t_k}^H-B_{t_{k-1}}^H\right) + a\sigma\sum_{k=1}^n\left(\int_0^{t_k}Y_sds + \int_0^{t_{k-1}}Y_sds\right) \left(B_{t_k}^H-B_{t_{k-1}}^H\right)
\\
&\quad\quad+\sigma^2\sum_{k=1}^n \left(B_{t_k}^H+B_{t_{k-1}}^H\right) \left(B_{t_k}^H-B_{t_{k-1}}^H\right)
\\
&\quad\rightarrow 2\sigma \int_0^t\bigg(\sqrt{X_0}
+ a\int_0^sY_udu+\sigma B_s^H\bigg)\circ dB_s^H
= \tilde\sigma \int_0^t\sqrt{X_s}\circ dB_s^H,\quad \Delta t\rightarrow 0.
\end{split}
\end{equation}

Thus, the fractional Cox\,--\,Ingersoll\,--\,Ross process, introduced by Definition~\ref{def:CIR}, satisfies the stochastic differential equation
\begin{equation}
\begin{gathered}\label{stratonovich equation}
X_t=X_0+\tilde a\int_0^tX_s ds+\tilde\sigma\int_0^t\sqrt{X_s}\circ dB_s^H,
\end{gathered}
\end{equation}
where $\int_0^t\sqrt{X_s}\circ dB_s^H$ is a pathwise Stratonovich integral.

Let us make a few remarks on the resulting stochastic differential equation.

\begin{remark}
	The passage to the limit in \eqref{stratonovich sums} is correct, since the left-hand side of the equation \eqref{integral sums} does not depend on the partition, and the limit in formula \eqref{lebesgue sums} exists as the pathwise Lebesgue integral, therefore the corresponding pathwise Stratonovich integral also exists.
\end{remark}

\begin{remark}
	If $H>2/3$ the solution of the equation \eqref{stratonovich equation} coincides with the solution of the equation \eqref{main sde}, since the corresponding pathwise Stratonovich integral and the pathwise Riemann\,--\,Stieltjes integral coincide.
\end{remark}

\section{Probability of hitting zero by the fractional Ornstein\,--\,Uhlenbeck process}
\label{sec3}

Let us investigate the probability that the first zero hitting moment $\tau$ of the fractional Ornstein-Uhlenbeck process, which is a solution of equation~\eqref{fOU}, is finite.
According to \cite{CKM}, this solution can be written explicitly as
\begin{equation}\label{eq:fou-explicit}
    Y_t = e^{at} \left(Y_0 + \sigma\int_0^te^{-as}\,dB_s^H\right),
\end{equation}
where the integral with respect to fractional Brownian motion is the limit of Riemann\,--\,Stieltjes sums and can be defined by integration by parts:
\begin{equation}\label{eq:int}
     J_t \coloneqq\int_0^te^{-as}\,dB_s^H \coloneqq e^{-at} B_t^H + a\int_0^t e^{-as} B_s^H\,ds.
\end{equation}

From the formula~\eqref{eq:fou-explicit}, we see that the first zero hitting moment by the process $Y_t$ coincides with the first time the integral~\eqref{eq:int} reaches the level $-Y_0/\sigma$.
Note, that this integral is a normally distributed random variable with zero mean.
Therefore, due to the symmetry of the normal distribution, the probability that the integral~\eqref{eq:int} hits the negative level $-Y_0/\sigma$ coincides with the probability of reaching the positive level $Y_0/\sigma$.
Thus, the problem of studying the probability that integral $J_t$ hits the level $x>0$ in a finite time arises.
Naturally, the behavior of this integral essentially depends on the sign of the parameter $a\in\R$.

Let us consider two cases.

\subsection*{The case $a\le0$}
\begin{proposition}\label{prop0}
	If $a<0$ then
	\begin{equation*}
        \P\left(\limsup_{t\to\infty}J_t = +\infty\right) = 1.
	\end{equation*}
\end{proposition}

\begin{proof}
It is known \cite{CKM}, that for $a<0$ the process
\begin{equation*}
	G_t = e^{a t} \int_{-\infty}^t e^{-a s}\,dB^H_s
\end{equation*}
is Gaussian, stationary, and ergodic. By the ergodic theorem, for any $x\in\R$,
\begin{equation*}
    \frac1n \sum_{k=1}^n \ind_{\set{G_k>x}}
    \to\E\ind_{\set{G_0>x}} = \P(G_0>x)>0
    \quad \text{a.\,s., }n\to\infty.
\end{equation*}
Therefore,
\begin{equation*}
	\sum_{k=1}^\infty \ind_{\set{G_k>x}} = +\infty
    \quad\text{a.\,s.}
\end{equation*}
This means that the event $\set{G_k>x}$ happens infinitely often. Hence,
\begin{equation*}
	\limsup_{t\to\infty}G_t = +\infty \quad\text{a.\,s.}
\end{equation*}
Then
\begin{equation*}
    \limsup_{t\to\infty} J_t
    = \limsup_{t\to\infty} \left(e^{-a t}G_t-G_0\right)
	= +\infty \quad\text{a.\,s.}
    \qedhere
\end{equation*}
\end{proof}

\begin{remark}\label{rem:a=0}
Proposition \ref{prop0} remains true in case $a=0$, as $J_t=B^H_t$.
More detailed results on the behavior of the supremum and the probabilities of level hitting for the fBm can be found, for example, in \cite{decr-nual,lei-nual,Molchan2000,Nourdin-book}.
\end{remark}

\subsection*{The case $a>0$}

According to Corollary~\ref{cor:1},
\[
    V_{t}^{2} \coloneqq \var J_t = H\int_{0}^{t}z^{2H-1} \left(e^{-2a t + a z} + e^{-a z}\right) dz.
\]
The derivative of $V_{t}^{2}$ is equal to
\[
    \frac{d}{dt}V_{t}^{2} = 2H\left(t^{2H-1}e^{-a t} - ae^{-2a t} \int_{0}^{t}z^{2H-1}e^{a z}\,dz\right)  .
\]

Since the second summand under brackets is exponentially smaller than the first one, there exists $t(a)$ such that the derivative is positive for all $t\geq t(a)$.
Notice that
\[
    \lim_{t\rightarrow\infty}V_{t}^{2} = H\int_{0}^{\infty}z^{2H-1}e^{-a z}dz = \frac{H\Gamma(2H)}{a^{2H}}.
\]	

Introduce Gaussian process
\[
	Z_{t} = J_{t/(1-t)}, \quad t\in[0,1],
\]
defining $Z_{1}=J_{\infty}$.
We have for the derivative of its variance $v_{t}^{2}$
\begin{align}
    \frac{d}{dt}v_{t}^{2} &=\left.  \frac{d}{ds}V_{s}^{2}\right\vert_{s=t/(1-t)} \left(  \frac{t}{1-t}\right)'\nonumber\\
    &=2H\frac{\bigl(t/(1-t)\bigr)^{2H-1}e^{-a t/(1-t)} - a e^{-2a t/(1-t)} \int_{0}^{t/(1-t)} z^{2H-1} e^{a z}\,dz}{(1-t)^{2}}.
    \label{derivative variance}
\end{align}

It exists and tends to zero as $t\rightarrow1$.

Since the exponential multipliers, the second derivative also tends to zero as $t\rightarrow1$.
Now let us use the form of covariance function from Corollary~\ref{cor:1} (we mean everywhere that $s<t$).
We have after reducing similar terms:
\begin{align*}
    \E(J_{t}-J_{s})^{2} &= V_{t}^{2}+V_{s}^{2}-2\cov(J_s,J_t)
	\\
    &= He^{-2a t} \int_{0}^{t}z^{2H-1}e^{-az}dz + H\int_{0}^{t} z^{2H-1}e^{-a z}dz
    +He^{-2a s}\int_{0}^{s}z^{2H-1}e^{a z}dz
    \\
    &\quad+H\int_{0}^{s}z^{2H-1}e^{-a z}dz
    +He^{-2a s}\int_{0}^{t-s}z^{2H-1}e^{-a z}dz
    \\
    &\quad-He^{-2a t}\int_{t-s}^{t}z^{2H-1}e^{a z}dz
	+H\int_{s}^{t}z^{2H-1}e^{-a z}dz
    \\
    &\quad-He^{-2a s}\int_{0}^{s}z^{2H-1}e^{a z}dz
    -2H\int_{0}^{t}z^{2H-1}e^{-a z}dz
	\\
    &=He^{-2a t}\int_{0}^{t-s}z^{2H-1}e^{a z}dz + He^{-2a s}\int_{0}^{t-s}z^{2H-1}e^{-a z}dz.
\end{align*}
For $Z_{t}$:
\begin{align*}
    \E(Z_{t}-Z_{s})^{2} &= He^{-2a t/(1-t)} \int_{0}^{(t-s)/(1-t)(1-s)} z^{2H-1} e^{a z}\,dz
	\\
    &\quad+He^{-2a s/(1-s)} \int_{0}^{(t-s)/(1-t)(1-s)} z^{2H-1} e^{-a z}\,dz.
\end{align*}
Further, for any $s\in[0,1)$
\begin{align*}
    \frac{t-s}{(1-t)(1-s)} &= \frac{t-s}{(1-s)^{2}} + \frac{(t-s)^{2}}{(1-t)(1-s)};
    \\
    e^{-2a t/(1-t)} &= e^{-2a s/(1-s)} e^{-2a(t-s)/(1-t)(1-s)}
	\\*
    &= e^{-2a s/(1-s)} \left(1 + \frac{-2a(t-s)}{(1-s)^{2}} + \frac{-4a(t-s)^{2}}{(1-s)^{2}} + \frac{O\left((t-s)^{3}\right)}{(1-s)^{2}}\right)
\end{align*}
as $t\downarrow s$.
Moreover,
\[
    \int_{0}^{(t-s)/(1-t)(1-s)}z^{2H-1}e^{\pm a z}dz = \frac{1}{2H}\frac{(t-s)^{2H} + O\left((t-s)^{2H+1}\right)}{(1-s)^{2H}}
\]
as $t\downarrow s$.
Hence	
\begin{equation}
    \E(Z_{t}-Z_{s})^{2} = He^{-2a s/(1-s)}\left(  \frac{(t-s)^{2H}}{H(1-s)^{2H}} + \frac{O\left((t-s)^{2H+1}\right)}{(1-s)^{2H}}\right)
    \label{natural metrics}
\end{equation}
as $t\downarrow s$.
By this way it can also be obtained that for any $s\in[0,1]$,
\begin{equation}
    \limsup_{t-s\downarrow0} \frac{\E(Z_{t}-Z_{s})^{2}}{(t-s)^{2H}} \leq H\max_{s\in[0,1]}(1-s)^{-2}e^{-2a s/(1-s)}.
    \label{uniform estim}
\end{equation}
From \eqref{derivative variance} and \eqref{uniform estim} it follows that
\begin{equation}
    \P\left(\sup_{t\geq0}J_{t}<\infty\right) = \P\left(\max_{t\in[0,1]}Z_{t}<\infty\right) = 1.
    \label{a.s. bound}
\end{equation}
Further, from Theorem D.4 of \cite{Pit} and (\ref{derivative variance},\ref{uniform estim}) we obtain the following result.
\begin{proposition}\label{prop2}
    There exists a constant $C$ such that for any $x>0$%
	\begin{equation}
	\label{eq:probab}
        \P\left(\sup_{t\geq0}J_{t}\geq x\right) =\P\left(\max_{t\in[0,1]}Z_{t}\geq x\right)
        \leq C x^{\frac{1}{H}-1} \exp\left(-\frac{x^{2}}{2v^{2}}\right),
	\end{equation}
    with
    $v^{2} = \sup_{t\geq0}V_{t}^{2} = \max_{t\in[0,1]}v_{t}^{2}<\infty$.		
\end{proposition}
	
Moreover, since $v_{t}$ is twice differentiable, using \eqref{uniform estim} we obtain that for some $c>0$ and for any $t-s$, that is small enough,
\[
\mathrm{Cov}(Z_{s}/v_{s},Z_{t}/v_{t})\geq1-c|t-s|^{2H}.
\]

Now we can use the Slepyan theorem (see \cite{Pit})
to bound the probability \eqref{eq:probab} above by the corresponding probability for the process $v_{t}U_{t}$, where $U_{t}$ is a stationary Gaussian process with zero mean and covariance function, which behavior near zero is $1-|t|^{2H}+o(|t|^{2H})$.
Then, using Theorem D.4 from \cite{Pit}, we get

\begin{proposition}\label{prop3}
    There exists a constant $C_{1}$ such that for an arbitrary $x>0$
    \[
        \P\left(\sup_{t\geq0}J_{t}\geq x\right)
        \leq C_{1} x^{\frac{1}{H}-2} \exp\left(-\frac{x^{2}}{2v^{2}}\right).
	\]
\end{proposition}
		
\begin{remark}
    It is easy to verify that $\max_{t\in[0,1]}v_{t}^{2} = v_{1}^{2} =  V_\infty^2=\frac{H\Gamma(2H)}{a^{2H}}$.
    Indeed, from \eqref{derivative variance} it follows that the point $t=1$ is a local maximum point.
    With $t=0$ we obtain that $v^2_0=V^2_0=0$. Therefore, it is enough to show that the function $v^2_t$ has no local extremum  at interior points of the interval $[0,1]$.
    If such points existed, then, according to \eqref{derivative variance}, they would satisfy the equation
    \[
        \left(\frac{t}{1-t}\right)^{2H-1}e^{-a t/(1-t)}-a e^{-2a t/(1-t)} \int_{0}^{t/(1-t)} z^{2H-1} e^{a z}\,dz = 0, \quad t\in(0,1).
     \]
    Denote $s=\frac{t}{1-t}$, then this equation takes the following form:
    \begin{equation}\label{eq:extrem}
        e^{-2a s}\left(s^{2H-1}e^{a s} - a \int_{0}^{s} z^{2H-1} e^{a z}\,dz\right) = 0, \quad s>0.
     \end{equation}
     Let us investigate the behavior of the function
     \[
        h(s)=s^{2H-1}e^{a s}-a \int_{0}^{s} z^{2H-1} e^{a z}\,dz
     \]
     with $s>0$.
     First, note that with $H=1/2$ the function $h(s)\equiv1$, so the equation \eqref{eq:extrem} has no roots.
     For $H\ne1/2$ let us find the derivative:
     \[
        \frac{d}{dt}\,h(t) = (2H-1) s^{2H-2} e^{a s}.
     \]

     If $H<1/2$, then the function $h$ is strictly decreasing, therefore (as $a>0$) the left-hand side of \eqref{eq:extrem} is also strictly decreasing and tends to zero as $s\to\infty$. Thus, the equation \eqref{eq:extrem} has no roots on $(0,+\infty)$.

     If $H>1/2$, then $h$ is strictly increasing, and $h(0)=0$, so, $h(s)>0$ if $s>0$, and the equation \eqref{eq:extrem} either has no roots on $(0,+\infty)$.
\end{remark}
	
\begin{remark}
For $1>t>s$ from \eqref{natural metrics} and \eqref{uniform estim} it follows that $\E(Z_{t}-Z_{s})^{2}$ exponentially tends to zero as
$s\to1$. It can be shown that all derivatives with respect to $s$ of this expectation also tend to zero as $s\to1$.
The same is true for $v_{s}^{2}$.
It means that Theorem D.3 from \cite{Pit} cannot be used directly to obtain the asymptotic behavior of  $\P(\sup_{t\geq0}Z_{t}\geq x)$ as $x\rightarrow\infty$. However, this asymptotic behavior can be found by the methods used to prove the mentioned theorem.
\end{remark}

\medskip

Let us return to the question of the first zero hitting moment $\tau$ of the fOU process~\eqref{fOU}.
From the considerations given at the beginning of this section, propositions \ref{prop0}, \ref{prop3}, and Remark ~\ref{rem:a=0}, we obtain the following result.

\begin{theorem}
    \begin{enumerate}
    \item
        If $a\le0$, then $\P(\tau<\infty)=1$.
    \item
        If $a>0$, then $\P(\tau<\infty)\in(0,1)$, and we have the upper bound
        \[
            \P(\tau<\infty) \le C_{1} \left(\frac{Y_0}{\sigma}\right)^{\frac{1}{H}-2}  \exp\left(-\frac{a^{2H} Y_0^2}{\sigma^2\Gamma(2H+1)}\right),
        \]
        where $C_1>0$ is a constant.
    \end{enumerate}
\end{theorem}

\section{Appendix. Covariance function of the fractional Ornstein-Uhlenbeck process}
\label{sec4}

Consider the fractional Ornstein-Uhlenbeck process $Y$, which is the solution of the equation~\eqref{fOU} with the initial condition $Y_0=y_0\in\R$.
According to \eqref{eq:fou-explicit}--\eqref{eq:int}, this solution has the following form:
\begin{equation}
\label{eq: 3}
Y_{t}=y_{0}e^{a t}+a\sigma e^{a t}\int_0^te^{-a s}B_{s}^{H}ds + \sigma B_t^H,
\quad t\ge 0.
\end{equation}

\begin{proposition} \label{prop1}.
    Let $t\ge s\ge0$. Then covariance function of the fractional Ornstein\,--\,Uhlenbeck process \eqref{eq: 3} can be represented in the following form:
	\begin{equation}\label{eq:covOU}
    \begin{split}
        R_H(t,s) &= \frac{H\sigma^2}{2}\left(-e^{a t-a s}\int_0^{t-s}e^{-a z}z^{2H-1}dz + e^{-a t+a s}\int_{t-s}^{t}e^{a z}z^{2H-1}dz\right.
        \\
        &\quad-e^{a t+a s}\int_s^te^{-a z}z^{2H-1}dz + e^{a t-a s}\int_0^se^{a z}z^{2H-1}dz
        \\
        &\quad+\left.2e^{a t+a s}\int_0^te^{-a z}z^{2H-1}dz \right).
	\end{split}
    \end{equation}
\end{proposition}

\begin{proof}
Using \eqref{eq: 3} and the form of fBm covariance function, we have
\begin{align*}
    R_H(t,s) &= \E \left[\left(Y_{t}-y_{0}e^{a t}\right) \left(Y_{s}-y_{0}e^{a s}\right)\right]
    \\
    &=\E \left[\left(a\sigma e^{a t}\int_0^te^{-a u}B_{u}^{H}du + \sigma B_t^H\right) \left(a \sigma e^{a s}\int_0^se^{-a v}B_{v}^{H}dv + \sigma B_s^H\right)\right]
    \\
    &=\frac{a\sigma^2}{2}e^{a t}\int_0^t e^{-a u } \left(u^{2H}+s^{2H}-|u-s|^{2H}\right) du
    \\
    &\quad+\frac{a\sigma^2}{2}e^{a s}\int_0^s e^{-a v } \left(v^{2H}+t^{2H}-|v-t|^{2H}\right) dv + \frac{\sigma^2}{2}\left(t^{2H}+s^{2H}-|t-s|^{2H}\right)
    \\
    &\quad+\frac{a^2\sigma^2}{2}e^{a t + a s}\int_0^t\int_0^s e^{-a u - a v} \left(u^{2H}+v^{2H}-|u-v|^{2H}\right) du\,dv
    = \frac{\sigma^2}{2} \sum\limits_{n=1}^{10} I_{n},
\end{align*}
where
\begin{gather*}
    I_1 = ae^{a t}\int_0^te^{-a u}s^{2H}du,
    \quad
    I_2=ae^{a t}\int_0^te^{-a u}u^{2H}du,
    \quad
    I_3 = -ae^{a t}\int_0^te^{-a u}|u-s|^{2H}du,
    \\
    I_4 =  ae^{a s}\int_0^se^{-a v}t^{2H}dv,
    \quad
    I_5 = ae^{a s}\int_0^se^{-a v}v^{2H}dv,
    \quad
    I_6=-ae^{a s}\int_0^se^{-a v}(t-v)^{2H}dv,
    \\
    I_7=t^{2H}+s^{2H}-(t-s)^{2H},
    \quad
    I_8 = a^2e^{a t +a s} \int_0^te^{-a v}dv\int_0^se^{-a u}u^{2H}du,
    \\
    I_9=a^2 e^{a t +a s}\int_0^se^{-a u}du\int_0^te^{-a v}v^{2H}dv,
    \quad
    I_{10} = -a^2e^{a t + a s}\!\int_0^t\!\!\int_0^se^{-a u - a v}|u-v|^{2H}dudv.
\end{gather*}

The first two integrals are equal to
\[
    I_1 = s^{2H}\left(e^{a t}-1\right)
    \quad\text{and}\quad
    I_2 = -e^{a t}\int_0^t u^{2H} de^{-a u}
	= -t^{2H}+2He^{a t}\int_0^te^{-a u}u^{2H-1}du.
\]

Substituting the variables and integrating by parts we obtain
\begin{align*}
    I_3 &= -ae^{a t}\int_0^se^{-a u}(s-u)^{2H}du
	-ae^{a t}\int_s^te^{-a u}(u-s)^{2H}du
    \\
    &=-ae^{a t-a s}\int_0^se^{a z}z^{2H}dz
    -ae^{a t-a s}\int_0^{t-s}e^{-a z}z^{2H}dz
    \\
    &= -e^{a t-a s}\left(e^{a s}s^{2H} - 2H\int_0^se^{a z} z^{2H-1}dz - e^{-a (t-s)}(t-s)^{2H}\right.
    \\*
    &\quad+\left. 2H\int_0^{t-s}e^{-a z}z^{2H-1}dz\right)
    \\
    &=-e^{a t}s^{2H}+(t-s)^{2H} + 2He^{a t-a s}\int_0^se^{a z}z^{2H-1}dz
    - 2He^{a t-a s}\int_0^{t-s}e^{-a z}z^{2H-1}dz.
\end{align*}

Similarly, transforming $I_4$--$I_6$ we have:
\begin{gather*}
    I_4 =  t^{2H}\left(e^{a s}-1\right),
    \qquad
    I_5 = -s^{2H} + 2He^{a s}\int_0^se^{-av}v^{2H-1}dv,
    \\
    \begin{split}
    I_6&=-ae^{a s-a t}\int_{t-s}^{t}e^{a z}z^{2H}dz
    = -e^{a s-a t}\int_{t-s}^{t}z^{2H}de^{a z}
    \\
    &=-e^{a s}t^{2H} + (t-s)^{2H} + 2He^{a s - a t}\int_{t-s}^{t}e^{a z}z^{2H-1}dz.
    \end{split}
\end{gather*}

Then,
\[
    I_{8} = e^{a t+a s}\left(e^{-a t} -1\right)\int_0^s u^{2H}de^{-a u}
    =(1-e^{a t})s^{2H} - 2He^{a s}(1-e^{a t })\int_0^s e^{-a u}u^{2H-1}du,
\]
and similarly
\[
    I_{9} = (1-e^{a s})t^{2H} - 2He^{a t}(1-e^{a s }) \int_0^t e^{-a v}v^{2H-1}dv
\]

Now consider $I_{10}$.
After representing it as the sum of integrals, we have:
\begin{align*}
    I_{10} &= -a^2e^{a t + a s}\int_0^s\!\!\int_0^ve^{-a u - a v}(v-u)^{2H}dudv
    -a^2e^{a t + a s}\int_0^s\!\!\int_v^se^{-a u - a v}(u-v)^{2H}dudv
    \\
    &\quad-a^2e^{a t + a s}\int_s^t\!\!\int_0^se^{-a u - a v}(v-u)^{2H}dudv
    \\
    &= -2a^2e^{a t + a s}\int_0^s\!\!\int_0^ve^{-a u - a v}(v-u)^{2H}dudv
    -a^2e^{a t + a s}\int_s^t\!\!\int_0^se^{-a u - a v}(v-u)^{2H}dudv
    \\*
    &\eqqcolon  I_{10}'+I_{10}''.
\end{align*}

By substituting $v-u=z$, changing the order of integration, and integrating by parts, we obtain
\begin{align*}
    I_{10}' &= -2a^2e^{a t + a s}\int_0^s e^{-2a v} \int_0^ve^{a z}z^{2H}dzdv
    = -2a^2e^{a t + a s}\int_0^s e^{a z}z^{2H}\int_z^s e^{-2a v}dvdz
    \\
    &= -2a^2e^{a t + a s}\int_0^s e^{a z}z^{2H}\, \frac{e^{-2as}-e^{-2az}}{-2a}\,dz
    \\
    &= ae^{a t -a s}\left(\int_0^se^{a z}z^{2H}dz - \int_0^se^{-a u} u^{2H}du\right)
    \\
    &=e^{a t}s^{2H}
    -2He^{a t-a s}\int_0^se^{a z}z^{2H-1}dz + e^{at}s^{2H}
    - 2He^{a t +a s}\int_0^se^{-a u}u^{2H-1}du.
\end{align*}

To simplify $I_{10}''$ we need to consider two cases.

If $t>2s$, then, after substituting $v-u=z$ in the inner integral, changing the order of integration and integrating with respect to $v$, we get
\begin{align*}
    I_{10}'' &=-a^2 e^{a t + a s}\int_s^t\!\!\int_{v-s}^ve^{a z - 2a v}z^{2H}dzdv
    =-a^2e^{a t + a s}\left(\int_0^s\!\!\int_s^{z+s}e^{a z - 2a v}z^{2H}dvdz\right.
    \\*
    &\quad+\left.\int_s^{t-s}\!\!\int_z^{z+s}e^{a z - 2a v}z^{2H}dvdz+\int_{t-s}^t\!\int_z^{t}e^{a z - 2a v}z^{2H}dvdz\right)
    \\
    &=-a^2e^{a t + a s}\left(\int_0^se^{a z} z^{2H}\,\frac{e^{-2a(z+s)}-e^{-2as}}{-2a}\,dz \right.
    \\*
    &\quad+\left.\int_s^{t-s}e^{a z}z^{2H} \frac{e^{-2a(z+s)}-e^{-2az}}{-2a}\,dz
    +\int_{t-s}^te^{a z}z^{2H} \frac{e^{-2at}-e^{-2az}}{-2a}\,dz\right)
    \\
    &= \frac{a}{2}e^{a t - a s}\int_0^{t-s}e^{-a z}z^{2H}dz - \frac{a}{2}e^{a t +a s}\int_s^te^{-a z}z^{2H}dz
    - \frac{a}{2}e^{a t-a s}\int_0^se^{a z}z^{2H}dz
    \\
    &\quad+ \frac{a}{2}e^{a s - a t}\int_{t-s}^te^{a z}z^{2H}dz.
    \end{align*}

    After integrating by parts each of the four integrals and reducing similar terms, we get the following equality:
    \begin{align*}
    I_{10}'' &= -He^{a s-a t}\int_{t-s}^te^{a z}z^{2H-1}dz + He^{a t - a s}\int_0^se^{a z}z^{2H-1}dz
    - He^{a t + a s}\int_s^te^{-a z}z^{2H-1}dz
    \\
    &\quad+He^{a t - a s}\int_0^{t-s}e^{-a z}z^{2H-1}dz - e^{a t}s^{2H} + e^{a s}t^{2H}-(t-s)^{2H}.
\end{align*}
The validity of the last formula for $s<t<2s$ can be checked similarly.

Summing all the terms, we obtain \eqref{eq:covOU}.
\end{proof}

\begin{remark}
    For $s=t$ from \ref{prop1} we get the formula
    \begin{align*}
        \var Y_t &= H\sigma^2\int_0^t z^{2H-1} \left(e^{a z} + e^{2a t-a z}\right)dz,
	\end{align*}
    that was proved in \cite[Lemma~A.1]{KMR}.
\end{remark}

\begin{remark}
For arbitrary $t,s\in\R^+$ we have
\begin{align*}
    R_H(t,s) &= \frac{H\sigma^2}{2}\left(-e^{a\abs{t-s}} \int_0^{\abs{t-s}}e^{-a z}z^{2H-1}dz + e^{-a\abs{t-s}}\int_{\abs{t-s}}^{\max\set{t,s}}e^{a z}z^{2H-1}dz\right.
	\\
    &\quad-e^{a (t+s)}\int_{\min\set{t,s}}^{\max\set{t,s}}e^{-a z}z^{2H-1}dz + e^{a \abs{t-s}}\int_0^{\min\set{t,s}}e^{a z}z^{2H-1}dz
	\\
    &\quad+\left.2e^{a (t+s)}\int_0^{\max\set{t,s}} e^{-a z}z^{2H-1}dz \right).
	\end{align*}
\end{remark}

\begin{corollary}\label{alt form of fCIR cov function}
	Let $H\in(0,\frac{1}{2})\cup(\frac{1}{2},1], a<0$ and $t>0$ be fixed. Then the covariance function ot the Ornstein-Uhlenbeck process can be represented in the form
	\begin{equation}\label{eq: alt form of fCIR cov function}
	\begin{split}
	R_H(t+s,t) = &\frac{\sigma^2 H(2H-1)}{2(-a)^{2H}}\bigg(e^{as}\int_1^{-as}e^y y^{2H-2}dy+e^{-as}\int_{-as}^{+\infty}e^{-y}y^{2H-2}dy
	\\
	 &-e^{at}\bigg[e^{-a(t+s)}\int_{-a(t+s)}^{+\infty}e^{-y}y^{2H-2}dy+e^{a(t+s)}\int_1^{-a(t+s)}e^y y^{2H-2}dy\bigg]\bigg)
	\\
	&+O(e^{as}),\quad s\rightarrow\infty
	\end{split}
	\end{equation}
\end{corollary}
\begin{proof}
	We can assume that $s$ is such that $-as>1$. From the formula \eqref{eq:covOU} after substitution $y=-az$ we get:
	\begin{equation*}
	\begin{split}
	 R_H(t+s,t)&=\frac{H\sigma^2}{2(-a)^{2H}}\bigg(-e^{as}\int_0^{1}e^{y}y^{2H-1}dy-e^{as}\int_1^{-as}e^{y}y^{2H-1}dy
	\\
	&\quad+e^{-as}\int_{-as}^{-a(t+s)} e^{-y}y^{2H-1}dy-e^{2at+as}\int_{-at}^{-a(t+s)} e^{y}y^{2H-1}dy
	\\
	&\quad+e^{as}\int_0^{-at} e^{-y}y^{2H-1}dy+2e^{2at+as}\int_0^{1}e^{y}y^{2H-1}dy
	\\
	 &\quad+2e^{2at+as}\int_1^{-a(t+s)}e^{y}y^{2H-1}dy\bigg)
	\\
	 &=\frac{H\sigma^2}{2(-a)^{2H}}\bigg(-e^{as}\int_1^{-as}e^{y}y^{2H-1}dy+e^{-as}\int_{-as}^{-a(t+s)} e^{-y}y^{2H-1}dy
	\\
	&\quad-e^{2at+as}\int_{-at}^{-a(t+s)} e^{y}y^{2H-1}dy+2e^{2at+as}\int_1^{-a(t+s)}e^{y}y^{2H-1}dy\bigg)+O(e^{as})
	\end{split}
	\end{equation*}
	
	Integrating by parts and reducing similar terms, we obtain \eqref{eq: alt form of fCIR cov function}.
\end{proof}

\begin{remark}
	The results above are consistent with the results of \cite{CKM}. Indeed, let $a<0, H\in(0,\frac{1}{2})\cup(\frac{1}{2},1]$.
	In the proof of \cite[Corollary 2.5]{CKM} it is shown that
	\begin{equation}\label{eq: cov-cov}
	\begin{gathered}
	R_H(t+s,t)=\cov(\tilde{Y}_t^H, \tilde{Y}_{t+s}^H)-e^{at}\cov(\tilde{Y}_0^H, \tilde{Y}_{t+s}^H)+O(e^{as}),\quad s\rightarrow\infty,
	\end{gathered}
	\end{equation}
	where
	\begin{equation*}
	\tilde{Y}_t^H:=\sigma\int_{-\infty}^t e^{a(t-u)}dB_u^H,
	\end{equation*}
	and in the proof of \cite[Theorem 2.3]{CKM} the following expression for $\cov(\tilde{Y}_t^H, \tilde{Y}_{t+s}^H)$ as $s\rightarrow\infty$ is obtained:
	\begin{equation}\label{Cheridito form}
	\begin{split}
	\cov(\tilde{Y}_t^H, \tilde{Y}_{t+s}^H)=&\frac{\sigma^2}{2(-a)^{2H}}H(2H-1)\times
	\\
	 &\times\left(e^{as}\int_1^{-as}e^yy^{2H-2}dy+e^{-as}\int_{-as}^{\infty}e^{-y}y^{2H-2}dy\right)+O(e^{as})
	\end{split}
	\end{equation}
	From \eqref{Cheridito form} and \eqref{eq: cov-cov} we can get an expression than coincides with \eqref{eq: alt form of fCIR cov function}.
\end{remark}

Note, that the integral $J_t$, defined by \eqref{eq:int}, is equal to $e^{-a t}Y_t$, where $Y_t$ is the process \eqref{eq: 3} with parameters $y_0=0$, $\sigma=1$.
Thus, we have the following result
\begin{corollary}\label{cor:1}
    If $s\le t$
    \begin{align*}
        \cov(J_s,J_t) = &-\frac{H}{2}e^{-2a s}\int_{0}^{t-s}z^{2H-1}e^{-a z} dz + \frac{H}{2}e^{-2a t} \int_{t-s}^{t}z^{2H-1}e^{a z}dz
        \\
        &-\frac{H}{2}\int_{s}^{t}z^{2H-1}e^{-a z}dz +\frac{H}{2}e^{-2a s} \int_{0}^{s}z^{2H-1}e^{a z}dz
        \\
        &+H\int_{0}^{t}z^{2H-1}e^{-a z}dz.
	\end{align*}
    In particular,
    \[
        \var J_{t} = H\int_{0}^{t}z^{2H-1} \left(e^{az-2at} +e^{-a z}\right) dz.
	\]
\end{corollary}

\end{document}